\documentclass[12pt, reqno]{amsart}
\usepackage{amsmath, amsthm, amscd, amsfonts, amssymb, graphicx, color}

\newtheorem{theorem}{Theorem}[section]

\newtheorem{proposition}[theorem]{Proposition}
\newtheorem{corollary}[theorem]{Corollary}
\theoremstyle{definition}

\theoremstyle{remark}
\newtheorem{remark}[theorem]{Remark}
\numberwithin{equation}{section}

\def\cH{\mathcal{H}}

\def\bh{\mathbb{B}(\cH)}

\def\sh{\mathcal S(\cH)}

\begin{document}

\title{\vspace*{0cm}\textbf{A Gr\"uss type operator inequality}}

\author[T. Bottazzi, C. Conde]{T. Bottazzi$^1$ and C. Conde$^{1,2}$}

\address{$^1$Instituto Argentino de Matem\'atica ``Alberto P. Calder\'on", Saavedra 15 3º piso, (C1083ACA) Buenos Aires, Argentina}

\address{$^2$Instituto de Ciencias, Universidad Nacional de Gral. Sarmiento, J.
M. Gutierrez 1150, (B1613GSX) Los Polvorines, Argentina}

\email{tpbottaz@ungs.edu.ar}
\email{cconde@ungs.edu.ar}



\begin{abstract} 
In [P. Renaud, \textit{A matrix formulation of Gr\"{u}ss inequality,}  Linear Algebra Appl. \textbf{335} (2001), 95--100] it was proved an operator inequality involving the usual trace functional. In this article, we give a refinement of such result and we answer positively the Renaud's open problem.
\end{abstract}

\maketitle

\section{Introduction}

In 1935, Gr\"{u}ss \cite{Gr} obtained the following inequality: if $f, 
g$ are integrable real functions on $[a, b]$ and there exist real constant $\alpha, \beta, \gamma, \delta$ such that
$\alpha\leq f(x)\leq \beta, \gamma\leq g(x)\leq \delta$ for all $x\in [a, b]$ then 
\begin{equation}\label{Gruss}
\left |\frac{1}{b-a}\int_a^b f(x)g(x)dx-\frac{1}{(b-a)^2}\int_a^b f(x)dx\int_a^b g(x)dx\right|\leq \frac14 (\beta-\alpha)(\delta-\gamma),
\end{equation}
and the inequality is sharp, in the sense that the constant $\frac 14$ cannot be replaced by a smaller one. This inequality has been investigated, applied and generalized
by many mathematicians in different areas of mathematics, such as inner product spaces,
quadrature formulae, finite Fourier transforms,  linear functionals, etc.

Along this work $\mathcal{H}$ denotes a (complex, separable) Hilbert
space with inner product $\langle \cdot,\cdot \rangle$. Let
$(\mathbb{B}(\cH), \|\cdot\|)$ be the $C^*$-algebra of all
bounded linear operators acting on  $(\cH,\langle
\cdot,\cdot\rangle)$ with the uniform norm.  We denote by $Id$ the identity operator, and for any $A\in \bh$ we consider $A^*$ its adjoint and $|A| = (A^*A)^{\frac 12}$ the absolute value of $A$.
By $\bh^+$ we denote the
cone of positive  operators of $\bh$, i.e. $\bh^+ := \{T \in \bh : \langle Th, h\rangle\geq 0 \: \forall h\in \cH\}.$ In the case when $\dim \cH = n$, we
identify $\mathbb{B}(\cH)$ with the full matrix algebra
$\mathcal{M}_n$ of all $n\times n$ matrices with entries in the
complex field $\mathbb{C}$. 
For each $T\in \bh$, we denote its spectrum  by $\sigma(T)$, that is,
$\sigma(T) =\{\lambda \in \mathbb C : T -\lambda Id \: \textrm{is  not invertible} \}$ and a complex
number $\lambda\in \mathbb{C}$ is said to be in the approximate point spectrum  of
the operator $T$, and we denote by $\sigma_{ap}(T)$, if there is a sequence $\{x_{n}\}$ of unit vectors satisfying $(T-
\lambda)x_{n} \to 0.$

For each operator $T$ we consider
$$
r(T)=\sup\{|\lambda|: \lambda \in \sigma(T)\} \qquad \qquad \textrm{spectral radius of }T,
$$
 $$
W(T)= \{\langle Th, h \rangle: \|h\|=1 \} \qquad \qquad \textrm{numerical range of }T
$$
and 
$$
w(T)=\sup\{ |\lambda|: \lambda \in W(T)\} \qquad \qquad \textrm{numerical radius of }T.
$$

Recall that for all $T \in \bh$, $r(T)\leq w(T)\leq \|T\| \leq 2w(T)$,  $\sigma(T)\subseteq \overline{W(T)}$ and by the Toeplitz-Hausdorff's Theorem  $W(T)$ is convex.

Renaud \cite{Re} gave a bounded linear operator analogue of Gr\"{u}ss inequality
by replacing integrable functions by operators and the integration by a trace function
as follows:  let $A, T\in \bh$,  suppose that $W (A)$ and $W (T)$ are contained in disks of radii $R_A$ and
$R_T$, respectively. Then for any positive trace class operator $P$  with $tr(P ) = 1$ holds
\begin{equation}\label{Renaud}
 |tr(PAT)-tr(PA)tr(PT)|\leq 4R_AR_T,
\end{equation}
and if  $A$ and $T$ are normal (i.e. $T^*T=TT^*$), the constant 4 can be replaced by 1. We can see can easily see that if $A=\alpha Id$ or $T=\beta Id$ with $\alpha, \beta \in \mathbb C$ then the left hand side is equal to zero.  In the same article, Renaud proposed the following open problem: to characterise $k(A, T)$, where
\begin{equation}\label{Renaudk}
 |tr(PAT)-tr(PA)tr(PT)|\leq k(A, T)R_AR_T,
\end{equation}
 with $1\leq k(A, T)\leq 4$. In particular, whether it depends on $A$ and $T$ separately, i.e.
whether we can write $k(A, T) = h(A)h(T)$, where $h(A), h(B)$ are suitably defined
constants. 

In this paper we give a positive answer to the open problem proposed by Renaud and we obtain an explicit formula for $k(A, T) = h(A)h(T)$. Also, we generalize the inequality \eqref{Renaud} for normal to transloid operators. 

\section{Preliminaries}

Let us begin with the notation and the necessary definitions. 

The set of compact operators in $\cH$ is denoted by $B_0(\cH)$.  If $T\in B_0(\cH)$  we denote by $\{s_n(T)\}$ the sequence of singular values of $T$, i.e., the eigenvalues of $|T|$ (decreasingly ordered). The notion of unitary invariant norms can be defined also for operators
on Hilbert spaces.
A norm $||| . |||$   that satisfies  the invariance property 
$
||| UXV|||=|||X|||.
$
If $\dim R(T) = 1$, then
$|||T||| = s_1(T)g(e_1) = g(e_1)\|T\|.$ By convention, we assume that $g(e_1)=1.$ If $x,y \in \cH$, then we denote $x\otimes y$ the rank one operator defined on $\cH$ by $(x\otimes y)(z)=\langle z, y\rangle x$ then $\|x\otimes y\|=\|x\|\|y\|=|||x\otimes y|||.$

The most known examples of unitary invariant norms are the Schatten $p$-norms  For $1\leq p <\infty$, let
$$\left\|T\right\|_p^p=\sum_n s_n(T)^p=tr\left|X\right|^p,
$$
and
$$B_p(\cH)=\{T\in \cH:\left\|T\right\|_p<\infty\},
$$
called the $p-Schatten \: class$ of $\bh$. That is the subset of compact operators with singular values in $l_p$.  The positive operators  with trace 1 are called density operator (or  states) and we denote this set by $\sh$. The ideal $B_2(\cH)$ is called  the Hilbert-Schmidt class and it is a Hilbert space with the inner product  $\langle S, T\rangle_2=tr(ST^*).$ On the theory of norm ideals and
their associated unitarily invariant norms, a reference for this subject is \cite{[GK]}.

An operator $A\in \bh$ is called normaloid if $r(A)=\left\|A\right\|=\omega(A)$. If $A-\mu Id$ is normaloid for all $\mu \in \mathbb C$, then $A$ is called transloid. 

Finally, for $A, T \in \bh$ and $P\in \sh$ we introduce the following notation
$$
 V_P(A, T)=tr(PAT)-tr(PA)tr(PT).
$$
In the particular case $T=A^*$ we get the variance  of $A$ respect to $P$. More precisely, Audenaert  in \cite{Au} consider the following notion, given $A, P\in\mathcal{M}_n, P\geq 0,   tr(P)=1$  the variance of $A$ respect to the matrix $P$
 $$
V_P(A)=tr(|A|^2P)-|tr(AP)|^2=V_P(A, A^*),
$$

Note that $V_P(A-\lambda Id)=V_P(A)$.
 Futhermore, he   showed that if $A\in \mathcal{M}_n$ then
\begin{equation}\label{Au}
\max\{tr(|A|^2P)-|tr(AP)|^2: P\in\mathcal{M}_n^+,  tr(P)=1\}=dist(A, \mathbb C Id)^2,
\end{equation}
and  the maximization over $P$ on the left hand side   can be restricted to density matrices of rank 1.

\section{Distance formulas and Renaud's inequality}

Let $A$ and $T$ linear bounded operators acting in $\cH$, the vector-function $A-\lambda T$ is known as the pencil generated by $A$ and $T$. Evidently there is at least one complex number $\lambda_0$ such that
$$
\|A-\lambda_0T\|=\inf_{\lambda\in \mathbb C} \|A-\lambda T\|.
$$
The number $\lambda_0$ is unique if $0\notin \sigma_{ap}(T)$ (or equivalently if $\inf\{\|Tx\|: \|x\|=1\}>0$). Different authors, following \cite{St}, called to this unique number as center of mass of $A$ respect to $T$ and we denote by $c(A, T)$ and when $T=Id$ we write $c(A)$. Following Paul,  for $A, T\in \bh$ such that $0 \notin \sigma_{ap}(T)$ we consider
\begin{equation}\label{paul}
M_T(A)=\sup_{\|x\|=1}\left[\|Ax\|^2-\frac{|\langle Ax, Tx\rangle|^2}{\langle Tx, Tx\rangle}\right]^{1/2}=\sup_{\|x\|=1}\left\|Ax-\frac{\langle Ax, Tx\rangle}{\langle Tx, Tx\rangle}Tx\right\|,
\end{equation}
in \cite{Pa}, he proved that $M_T(A)=dist(A,\mathbb CT)$. The unique minimizer  is characterized by the following conditions: there exists a sequence of unit vectores $\{x_n\}$ such that
$$
\|(A-\lambda_0T)x_n\|\to \|A-\lambda_0T\| \qquad \textrm{and} \qquad \langle (A-\lambda_0T)x_n, x_n\rangle \to0.
$$
In \cite{Ge}, Gevorgyan proved that 
\begin{equation}\label{centermass}
c(A, T)=\lim\limits_{n \to \infty}\frac{\langle Ay_n, Ty_n\rangle}{\langle Ty_n, Ty_n\rangle},
\end{equation}
where $\{y_n\}$ is a sequence of unit vectores which approximate the supremum in \eqref{paul}. In the particular case that $T=Id$ and $A$ is a Hermitian operator then it is easy to see that
\begin{eqnarray}
\min_{\lambda\in \mathbb C} \|A-\lambda Id\|= \frac{\lambda_{max}(A)-\lambda_{min}(A)}{2}, 
\end{eqnarray}
where  $\lambda_{max}(A)$ (resp. $\lambda_{max}(A)$) denotes the maximum (resp. mínimum) eigenvalue of $A$. Observe that the minumumis attained at 
$$
c(A)=  \frac{\lambda_{max}(A)+\lambda_{min}(A)}{2}.
$$


%
%

%
%

We recall other formulas that express the distance from $A$ to the one-dimensional subspace $\mathbb C T$. Then
\begin{equation}
dist(A,\mathbb CT)=\sup\{|\langle Ax, y\rangle| : \|x\|=\|y\|=1,\: \langle Tx, y\rangle=0\},
\end{equation}
if $A, T\in \bh$ and $0\notin \sigma_{ap}(T)$. In the particular case, where $T=Id$ we get

\begin{eqnarray}
dist(A,\mathbb CId)&=&\frac 12\sup\{\|AX-XA\|: X\in\bh, \|X\|=1\} \nonumber \\
&=& \sup \{\|(Id-Q)AQ\|:  Q\:  \textrm{is a rank one projection}\} \nonumber\\
&=& \sup \{\|(Id-Q)AQ\|:  Q\:  \textrm{is a  projection}\}.
\end{eqnarray}

In the following statement we present a new proof of the relation between the variance of $A$ respect to $P$ and the distance from $A$ to the unidimensional subspace $\mathbb C Id$.

 \begin{proposition}\label{newdragomir}
Let $A\in \bh$ and  $P\in\sh$ then
\begin{eqnarray}
tr(|A|^2P)-|tr(AP)|^2&=&\|AP^{1/2}\|_2^2- |\langle AP^{1/2}, P^{1/2} \rangle_2|^2 \nonumber \\
&=&\| AP^{1/2}- \langle AP^{1/2}, P^{1/2} \rangle_2 P^{1/2}\|_2^2\nonumber \\
&=& \min_{\lambda\in \mathbb C} \|AP^{1/2}-\lambda P^{1/2}\|_2^2\leq \min_{\lambda\in \mathbb C} \|A-\lambda Id\|. \nonumber \
\end{eqnarray}
\end{proposition}
\begin{proof}
These inequalities are simple consequences from  following general statement for any Hilbert space $\cH$: let $x,y \in \cH$ with $y\neq 0$ then
\begin{equation}\label{Hilbert}
\inf_{\lambda \in \mathbb C} \|x-\lambda y\|^2=\frac{\|x\|^2 \|y\|^2- |\langle x,y\rangle|^2}{\|y\|^2}. \nonumber\
\end{equation}

\end{proof}

The following statement is an extension of the Audenaert's formula to infinite dimension.

\begin{remark}[Audenaert's formula for infinite dimensional spaces]
We exhibit that the equality \eqref{Au} holds in infinite dimensional context, that is for $A\in \bh$ holds
\begin{equation}\label{Auinf}
\sup\{[ tr(|A|^2P)-|tr(AP)|^2]^{1/2}: P\in\sh\}=dist(A, \mathbb C Id).
\end{equation}
First, we obtain this equality from a Prasanna's result in \cite{Pr}. Indeed, note that 
\begin{eqnarray}
dist(A, \mathbb C Id)^2&=&\sup_{\|x\|=1} \|Ax\|^2-|\langle Ax, x\rangle|^2\nonumber\\
&\leq &\sup\{tr(|A|^2P)-|tr(AP)|^2: P\in\sh\}\nonumber \\
&\leq& dist(A, \mathbb C Id)^2.\nonumber \
\end{eqnarray}
On the other hand, another way to prove \eqref{Auinf} is to reduce the problem to finite dimension and use the classical Audenaert's formula. Now we give the idea of this proof.

For the sake of clarity, we denote
$$
m:=\min_{\lambda\in \mathbb C} \|A-\lambda Id\|
$$
and 
$$
M:=\sup\{[ tr(|A|^2P)-|tr(AP)|^2]^{1/2}: P\in \sh\}.
$$
By Proposition \ref{newdragomir} we have that $M\leq m$. Suppose  by contradiction that $M < m$ then there exists $\epsilon >0$ such that
\begin{equation}\label{epsilon}
M<  \|A-\lambda Id\| -\epsilon,
\end{equation}
for any $\lambda \in  \mathbb C.$ By the equality \eqref{centermass}, we have that $c(A)\in  \overline{W(A)} $ and then $|c(A)|\leq w(A)$.
As any closed ball in the complex plane is a  compact set, we can find  $\lambda_1, ..., \lambda_m \in \cH$ such that
$$
B(0, \omega(A))\subseteq \cup_{j=1}^m \{ \lambda \in  \mathbb C: |\lambda - \lambda_j|<\frac{\epsilon}{2}\}.
$$
Now, we choose unit vectors $h_1, ..., h_m \in \cH$ with the following property: $\|(A-\lambda_jId)h_j\|>\|A-\lambda_jId\|-\frac{\epsilon}{2}.$ Let $\cH'=gen\{h_1, ..., h_m, Ah_1, ..., Ah_m\}$ and $n=\dim \cH'$. Applying \eqref{Au} to the compressions of $A$ and $Id$ respectively, we get
\begin{equation}\label{Aufinito}
dist(A', \mathbb C Id_n)=\max\{[ tr(|A'|^2P')-|tr(A'P')|^2]^{1/2}: P' \in\mathcal{M}_n^+, tr(P')=1\}=M'.
\end{equation}
One easily verifies that if $\lambda \in B(0, \omega(A))$ there exists $j\in \{1, ..., m\}$ such that 
\begin{eqnarray}\label{ineqbola}
\|A'-\lambda Id_n\|&>&\|A'-\lambda_jId_n\|-\frac{\epsilon}{2}\geq \|(A'-\lambda_jId_n)h_j\|-\frac{\epsilon}{2}\nonumber \\
&=&\|(A-\lambda_jId)h_j\|-\frac{\epsilon}{2}>\|A-\lambda_jId\|-\epsilon.
\end{eqnarray}
Thus, combining \eqref{epsilon} and \eqref{ineqbola} we get
\begin{eqnarray}
\min_{\lambda\in \mathbb C} \|A'-\lambda Id_n\|>M\geq M',
\end{eqnarray}
and we have here a contradiction with \eqref{Aufinito}, therefore $m=M.$ 

\end{remark}
The next result gives and upper bound for $V_P(A, T)$.
\begin{proposition}
Let $A, T\in \bh$ and $P\in \sh$. Then, for any $\lambda,\mu\in \mathbb{C}$ holds
\begin{eqnarray}\label{refRenuad1}
|V_P(A, T)|&\leq& \left(A-\lambda Id, A-\lambda Id\right)_{2,P}^{1/2}(T^*-\bar{\mu}Id,T^*-\bar{\mu}Id)_{2,P}^{1/2}-G_{Id}(A-\lambda Id,T^*-\bar{\mu}Id)\nonumber\\
&\leq& \left\|A-\lambda Id\right\|\left\|T-\mu Id\right\|-\left|tr\left(P(A-\lambda Id)\right)tr\left(P(T-\mu Id)\right)\right|,
\end{eqnarray}
with $G_{Id}(A-\lambda Id,T^*-\bar{\mu}Id)=\left|tr\left(P(A-\lambda Id)\right)tr\left(P(T-\mu Id)\right)\right|$.

Therefore,
\begin{equation}\label{refRenuad}
|V_P(A, T)| \leq \sup_{\widetilde{P}\in \sh} |tr(\widetilde{P}AT)-tr(\widetilde{P}A)tr(\widetilde{P}T)|\leq dist(A, \mathbb C Id)dist(T, \mathbb C Id).
\end{equation}
\end{proposition}
\begin{proof}
Define the following semi-inner product for $X,Y\in \bh$ and $P\in \sh$:
$$(X,Y)_{2,P}=\left\langle P^{1/2}X,P^{1/2}Y\right\rangle_2.$$
Following the proof given by Dragomir in [\cite{Dr2},Theorem 2], holds for any $E\in \bh$ such that $(E,E)_{2,P}=1$
\begin{eqnarray}\label{drago inner}
\left|(X,Y)_{2,P}-(X,E)_{2,P}(E,Y)_{2,P}\right|&\leq& (X,X)_{2,P}^{1/2}(Y,Y)_{2,P}^{1/2}-\left|(X,E)_{2,P}(E,Y)_{2,P}\right|.\nonumber\\
&=& (X,X)_{2,P}^{1/2}(Y,Y)_{2,P}^{1/2}-G_E(X,Y).
\end{eqnarray}
Since $(Id,Id)_{2,P}=1$, then
\begin{eqnarray}
|V_P(A, T)|&=&|V_P(A-\lambda Id, T-\mu Id)|\nonumber \\
&=& \left|(A-\lambda Id,(T-\mu Id)^*)_{2,P}-(A-\lambda Id,Id)_{2,P}(Id,(T-\mu Id)^*)_{2,P}\right|\nonumber\\
&\leq& \left(A-\lambda Id, A-\lambda Id\right)_{2,P}^{1/2}(T^*-\bar{\mu} Id,T^*-\bar{\mu} Id)_{2,P}^{1/2}-G_{Id}(A-\lambda Id,T^*-\bar{\mu}Id)\nonumber\\
&=& tr\left(P|(A-\lambda Id)^*|^2\right)^{1/2}tr\left(P|T-\mu Id|^2\right)^{1/2}-G_{Id}(A-\lambda Id,T^*-\bar{\mu}Id)\nonumber\\
&\leq& \left\||(A-\lambda Id)^*|^2\right\|^{1/2}\left\||T-\mu Id|^2\right\|^{1/2}-\left|tr\left(P(A-\lambda Id)\right)tr\left(P(T-\mu  Id)\right)\right|\nonumber\\
&=& \left\|A-\lambda Id\right\|\left\|T-\mu Id\right\|-\left|tr\left(P(A-\lambda Id)\right)tr\left(P(T-\mu Id)\right)\right|.
\end{eqnarray}
Therefore,
$$\sup_{\widetilde{P}\in \sh} |tr(\widetilde{P}AT)-tr(\widetilde{P}A)tr(\widetilde{P}T)|\leq dist(A, \mathbb C Id) dist(T, \mathbb C Id).$$
 
\end{proof}

\begin{remark}
If we define $V_P:\bh\times \bh \to\mathbb C$, $V_P(A, T):=tr(PAT)-tr(PA)tr(PT)$. 
Then $V_P$ is a bilinear function and by \eqref{refRenuad} a continuous mapping with $\|V_P\|\leq 1.$
\end{remark}
  Now, we give a new proof and a refinement  of \eqref{Renaud}. 
  
  \begin{proposition}
  Let $A, T \in \bh$ and we suppose that $W(A), W(T)$ are contained in closed disk $D(\lambda_0, R_A), D(\mu_0, R_T)$ respectively. Then for any $P\in \sh$
\begin{eqnarray}\label{refRe}
|tr(PAT)-tr(PA)tr(PT)|&\leq&\sup_{\widetilde{P}\in \sh} |tr(\widetilde{P}AT)-tr(\widetilde{P}A)tr(\widetilde{P}T)|\nonumber \\
&\leq& dist(A, \mathbb C Id)dist(T, \mathbb C Id)\nonumber\\
&\leq& \|A-\lambda_0 Id\| \|T-\mu_0 Id\|\nonumber\\
&\leq& 4w(A-\lambda_0 Id)w(T-\mu_0 Id)\nonumber\\
&\leq& 4R_AR_T.
\end{eqnarray}
In particular, if $A$ and $T$ are normal operators, we have
\begin{eqnarray}\label{refRenormal}
|tr(PAT)-tr(PA)tr(PT)|&\leq&\sup_{\widetilde{P}\in \sh} |tr(\widetilde{P}AT)-tr(\widetilde{P}A)tr(\widetilde{P}T)|\nonumber \\
&\leq& dist(A, \mathbb C Id)dist(T, \mathbb C Id)=r_Ar_T,
\end{eqnarray}
where  $r_S$ denotes the radius of the unique  smallest disc containing $\sigma(S)$ for any $S\in \bh$.
 \end{proposition}
 \begin{proof}
 The inequalities are consequence of \eqref{refRenuad}.
In the last inequality we use that $W(A-\lambda_0 Id)\subset D(0, R_A)$ and $W(T-\mu_0 Id)\subset D(0, R_T)$ respectively. 

On the other hand,  Bj\"{o}rck and  Thom\'ee \cite{BT} have shown that for a normal operator $A$
\begin{eqnarray}\label{BT}
dist(A, \mathbb C Id)=\sup_{\|x\|=1}( \|Ax\|^2-|\langle Ax, x\rangle|^2)^{1/2}= r_A,
\end{eqnarray}
and this completes the proof.

\end{proof}

\begin{remark}
From \eqref{refRenormal}, if we consider $A$ is a positive invertible operator, $T=A^{-1}$ and $P=x\otimes x$ with  $x\in\cH $ with $\|x\|=1$, then
\begin{eqnarray}
|tr(PAT)-tr(PA)tr(PT)|&=&|1-\langle Ax, x\rangle\langle A^{-1}x, x\rangle|\nonumber \\
&\leq& dist(A, \mathbb C Id)dist(A^{-1}, \mathbb C Id)=r_Ar_{A^{-1}},\nonumber\
\end{eqnarray}
i.e. we obtain the Kantorovich inequality for an operator $A$ acting on an infinite
dimensional Hilbert space $\cH$ with $0<m\leq A\leq M$.
\end{remark}
In 1972, Istratescu (\cite{I}) generalized the equality \eqref{BT}  to the transloid class operators,  then we have the following statement:
\begin{proposition}
Let $A, T \in \bh$ with $A$ and $T$ transloid operators then
\begin{eqnarray}\label{}
|tr(PAT)-tr(PA)tr(PT)|&\leq&\sup_{\widetilde{P}\in \sh} |tr(\widetilde{P}AT)-tr(\widetilde{P}A)tr(\widetilde{P}T)|\nonumber \\
&\leq& dist(A, \mathbb C Id)dist(T, \mathbb C Id)=r_Ar_T.
\end{eqnarray}
\end{proposition}
\begin{proof}
It follows from the same arguments in  the proof of inequality \eqref{refRenormal}.
\end{proof}
The previous proposition generalizes the Renaud's result for normal operators, since the classes of transloid and normal operators are related by the inclusion as follows
$$\text{normal}\subseteq \text{quasinormal} \subseteq \text{subnormal} \subseteq \text{hyponormal} \subseteq \text{transloid},$$
where at least the first inclusion is proper.

In the following statement we obtain a parametric refinement of \eqref{Renaud}. 

\begin{theorem} \label{mainRenaud}
Let $A, T \in \bh$ with $A, T \notin \mathbb C Id$ and suppose that $W(A), W(T)$ are contained in the closed disk $D(\lambda_0, R_A)$ and $D(\mu_0, R_T)$ respectively. Thus for any $P\in \sh$ we get
\begin{eqnarray}
|tr(PAT)-tr(PA)tr(PT)|&\leq&\sup_{\widetilde{P}\in \sh} |tr(\widetilde{P}AT)-tr(\widetilde{P}A)tr(\widetilde{P}T)|\nonumber \\
&\leq& dist(A, \mathbb C Id)dist(T, \mathbb C Id)\nonumber\\
&\leq& h_{\lambda}(A)h_{\mu}(T)\omega(A-\lambda_0 Id)\omega(T-\mu_0 Id)\nonumber\\
&\leq& h_{\lambda}(A)h_{\mu}(T)R_AR_T,
\end{eqnarray}
where 
\begin{eqnarray}
h_{\lambda}(A)=2(1-\lambda)+\lambda\frac{\|A-c(A) Id\|}{w(A-\lambda_0Id)},& h_{\mu}(T)=2(1-\mu)+\mu\frac{\|T-c(T) Id\|}{w(T-\mu_0Id)}\nonumber\
\end{eqnarray}
and $1\leq h_{\lambda}(A)h_{\mu}(T)\leq 4$, for any $\lambda,\mu\in [0,1]$.
\end{theorem}
\begin{proof} 
Let $\lambda\in[0,1]$. Then,
\begin{eqnarray}\label{Re2}
\|A-c(A)Id\|&\leq&\lambda \|A-c(A)Id\|+(1-\lambda) \|A-\lambda_0Id\| \nonumber\\
&\leq &\lambda \|A-c(A)Id\| + 2(1-\lambda) w(A-\lambda_0Id) \nonumber \\
&=&w(A-\lambda_0Id)\left(2(1-\lambda)+ \lambda\frac{\|A-c(A) Id\|}{w(A-\lambda_0Id)}\right)\nonumber\\
&=&w(A-\lambda_0Id)h_{\lambda}(A),\nonumber\
\end{eqnarray}
where  $1\leq h_{\lambda}(A)\leq 2$ since $\|A-c(A)Id\|\leq\|A-\lambda_0Id\|\leq 2w(A-\lambda_0Id).$
This inequality completes the proof.
\end{proof}

Note that the previous result gives a positive answer at the Renuad's open question \eqref{Renaudk}.

\begin{corollary}
Under the same notation as in Theorem \ref{mainRenaud}, if $A-\lambda_0Id$ and $T-\mu_0 Id$ are normaloid operators  then, for any $\lambda,\mu \in[0,1]$
\begin{eqnarray} \label{normaloides}
|tr(PAT)-tr(PA)tr(PT)|&\leq&\sup_{\widetilde{P}\in \sh} |tr(\widetilde{P}AT)-tr(\widetilde{P}A)tr(\widetilde{\textsc{P}}T)|\nonumber \\
&\leq& dist(A, \mathbb C Id)dist(T, \mathbb C Id)\nonumber\\
&\leq& (2-\lambda)(2-\mu)\omega(A-\lambda_0 Id)\omega(T-\mu_0 Id)\nonumber\\
&\leq& (2-\lambda)(2-\mu)R_AR_T.
\end{eqnarray}
\end{corollary}

\end{document}